\documentclass{article}
\usepackage{amsmath,amsfonts,amssymb,amsthm}
\usepackage{graphicx}
\usepackage{natbib}
\usepackage{url}

\newcommand{\bsx}{\boldsymbol{x}}
\newcommand{\rd}{\,\mathrm{d}}
\newcommand{\cf}{\mathcal{F}}
\newcommand{\e}{\mathbb{E}}

\newtheorem{theorem}{Theorem}

\title{A constraint on extensible quadrature rules}
\author{Art B. Owen\\Stanford University}
\date{January 2015}
\begin{document}
\maketitle

\begin{abstract}
When the worst case integration error in a family
of functions decays as $n^{-\alpha}$ for some $\alpha>1$
and simple averages along an extensible sequence match that
rate at a set of sample sizes $n_1<n_2<\dots<\infty$,
then these sample sizes must grow at least
geometrically. More precisely, $n_{k+1}/n_k\ge \rho$
must hold for a value $1<\rho<2$ that increases with $\alpha$.
This result always rules out arithmetic sequences but never rules
out sample size doubling. The same constraint holds in
a root mean square setting.
\end{abstract}

\section{Introduction}

For both Monte Carlo (MC) 
and quasi-Monte Carlo (QMC) sampling, 
\cite{sobo:1998} recommends that the number $n$ of
sample points should be increased geometrically,
not arithmetically.  Specifically, he recommends
using a sequence like $n_1$, $2n_1$, $4n_1$, etc., instead
of $n_1$, $2n_1$, $3n_1$ and so on.
For either MC or QMC, the estimate of an integral is a
simple unweighted average of the integrand at $n$
points.

In the case of Monte Carlo sampling with its
slow $n^{-1/2}$ convergence rate, if $n$ is too small
to get a good answer then taking $n+k$
sample points for $k\ll n$ is unlikely
to bring a meaningful improvement in accuracy.
\cite{sobo:1993:b} studies the correlations among Monte Carlo
estimates along sample sizes $n_k = 2^{k-1}n_1$.

In quasi-Monte Carlo sampling, much better convergence
rates are sometimes possible, depending particularly
on the smoothness and dimension of the problem space.
\cite{nova:wozn:2010} provide a comprehensive treatise
on error rates for numerical integration.
In favorable cases, a small change in $n$
might make a meaningful reduction in the error bound.
What we show here is that rate-optimal sample
sizes are widely spaced in those favorable cases.

\cite{sobo:1998} showed that
a better rate than $1/n$ cannot hold uniformly for all $n$.
\cite{hick:krit:kuo:nuye:2012} extended this
finding to arithmetic sequences of sample sizes.
They also showed that unequally weighted averages of function
evaluations can attain a better than $1/n$ rate at all values of $n$.
Suppose for instance that the first $n$ sample points are partitioned
into blocks of $n_j$ points, for $j=1,\dots,J$
where the estimate from block $j$ has error $O(n_j^{-\alpha})$.
Then weighting those within block estimates proportionally
to $n_j^a$, with $a\ge \alpha$, attains an error $O( (J/n)^\alpha)$. 
One can commonly arrange $J=O(\log n)$ and then
the error is $o(n^{-\alpha+\epsilon})$ for any $\epsilon>0$.

It remains interesting to consider equal weight rules.
For a complicated problem with weighted points from
multiple spaces, keeping track of the weights becomes
cumbersome. Also, there are quasi-Monte Carlo methods such as
higher order digital nets \citep{dick:2011}, that are
simultaneously rate optimal for more than one class
of functions, each with its own rate.  In such settings
we might want to use the same weights for multiple
integrands, but no single weighing might serve them all best.
Finally, if the constraints we find here on equal weight
rules are unpalatable for some given problem, it provides
motivation to switch to an unequally weighted rule.

An outline of this note is as follows.
Section~\ref{sec:sobols} presents the insight from
the appendix of \cite{sobo:1998} and the extension
by \cite{hick:krit:kuo:nuye:2012}.
If a QMC rule has worst case error $o(1/n)$ holding for all $n$, then
the points $\bsx_i$ must have some very strange
limit properties and the class of functions involved is odd
enough that we could do very well using only one
point $\bsx_n$ for very large $n$.  A generalization
of that argument shows that an $o(1/n)$ rate along
an arithmetic sequence of sample sizes raises similar problems.
Section~\ref{sec:geoworst}
shows that if quadrature error in a class $\cf$ of functions
has a worst case lower bound $mn^{-\alpha}$ for $\alpha>1$
and a specific sequence $\bsx_i$ is rate optimal at sample
sizes $n_1<n_2<\cdots$, then necessarily $n_{k+1}/n_k\ge \rho$
for some constant $1<\rho<2$ depending on $\alpha$ and on
how close the sequence comes to having the optimal constant.
Section~\ref{sec:georms} considers root mean squared error
for sequences of sample points incorporating some randomness.
The same constraints hold in this setting as in the worst
case setting.

\section{Ruling out arithmetic sequences}\label{sec:sobols}
It is not reasonable to expect a QMC rule to have
errors of size $o(1/n)$ for all values of $n\ge N_0$
for some $N_0>0$.   Intuitively, adding a single point
makes a change of order $1/n$ to the estimated integral.
Therefore two consecutive values of the integral estimate
are ordinarily an order of magnitude farther apart from each 
other than they could be by the triangle inequality, 
with the true integral value
making the third corner of the triangle.  This idea is made
precise in the appendix of \cite{sobo:1998}, as we outline here.

Let $\mu =\mu(f)=\int f(\bsx)\rd\bsx$ and
$\hat\mu_n = \hat\mu_n(f)=(1/n)\sum_{i=1}^n f(\bsx_i)$.
The integral is over $[0,1]^d$ and $\bsx_i\in[0,1]^d$
for $i\ge 1$.
Let
$$\eta_n = \eta_n(f) = \frac1n\sum_{i=1}^n f(\bsx_i)-\mu$$
and suppose that  the points $\bsx_i$ are chosen such
$$
\sup_{f\in\cf} |\eta_n(f)| \le \epsilon(n)
$$
where $\epsilon(n) = o(1/n)$, and $\cf$ is a class of integrands.
Sobol' considered $\epsilon(n) = O(n^{-\alpha})$ for some $\alpha>1$.
The classes $\cf$ that we study are usually balls
with respect to a seminorm, such as the standard
deviation in Monte Carlo and the total variation
in the sense of Hardy and Krause, for quasi-Monte Carlo.

Sobol' observed that
\begin{align*}
|f(\bsx_{n+1}) -\mu|& = |(n+1)\eta_{n+1} -n\eta_n| \\
& \le (n+1)\epsilon(n+1) + n\epsilon(n)\\
& \to 0
\end{align*}
as $n\to\infty$.
As a result $\lim_{n\to\infty} f(\bsx_i) = \mu(f)$
for all $f\in\cf$.  The set $\cf$ cannot be very
rich in this case.
As Sobol' noted, for $d=1$, if $\cf$ contains
$x$ it cannot also contain $x^2$.

If we had such a sequence $\bsx_i$ and
a class $\cf$ we might simply estimate $\mu$ by $f(\bsx_n)$
for one extremely large $n$, perhaps the largest
one for which we can compute $\bsx_n$.
Alternatively, when $f\in\cf$ are all known to be 
integrable and anti-symmetric functions on $[0,1]^d$ we can take
$\bsx_1 = (1/2,\dots,1/2)$ and have a zero error. 
This is the most favorable
case for antithetic sampling \citep{hamm:mort:1956}.
MC and QMC methods are ordinarily designed for
more general purpose use, and so such special settings
are of limited importance.

\cite{hick:krit:kuo:nuye:2012} 
extend Sobol's argument to show that we
should not expect $\epsilon( nk )  = o(1/n)$
as $n\to\infty$ for any integer $k\ge1$.
We would then have a class of functions $\cf$ with
$$
\lim_{n\to\infty} \frac1k\sum_{i=1}^kf(\bsx_{nk+i})\to\mu(f)
$$
for all $f\in\cf$. That is a very limited class,
and once again, we could solve the problem uniformly
over that class simply by taking $k$ points $\bsx_{nk+1},\dots
\bsx_{(n+1)k}$ for some very large $n$.

\section{Geometric spacing for the worst case setting}\label{sec:geoworst}

The class $\cf$ of real-valued functions
on $[0,1]^d$ has a superlinear worst case lower bound if
\begin{align}\label{eq:lowerbound}
\sup_{f\in\cf}\,
\biggl|  \frac1n\sum_{i=1}^nf(\bsx_i)-\mu(f)\biggr|
> mn^{-\alpha}
\end{align}
holds for some $\alpha>1$, $m>0$,
all $n\ge1$ and all $\bsx_i\in[0,1]^d$.
There is a uniformly rate optimal sequence
for this class, if
for some $\bsx_i\in[0,1]^d$ and a sequence of sample sizes
$n_1<n_2<\cdots$,
\begin{align}\label{eq:upperbound}
\sup_{f\in\cf}\,
\biggl|  \frac1n\sum_{i=1}^nf(\bsx_i)-\mu(f)\biggr|
\le Mn^{-\alpha},\quad\forall n\in\{n_1,n_2,\dots\}
\end{align}
holds, where $m\le M<\infty$.

The proof of Theorem~\ref{thm:worstcaseextension} below
makes use of some basic facts about fixed-point iterations.
Let $g(x)$ be a continuous function on the 
interval $[a,b]$ taking values in $[a,b]$. 
Then $g$ has at least one fixed point $x_*\in[a,b]$,
with $g(x_*)=x_*$. 
If also, $g$ has Lipschitz constant $\lambda <1$
for all $x\in[a,b]$, then the fixed point $x_*$ is unique. 
Now consider the fixed point iteration $x_{n+1}=g(x_{n})$. 
Under the Lipschitz condition, $x_n$ converges to $x_*$
from any $x_1\in[a,b]$.  These facts
are consequences of \citet[Theorem 4.2.1]{kelley1999iterative}.
When $g$ has derivative $g'$ with $0<g'(x)<1$ on $[a,b]$, then the convergence to $x_*$
is monotone: if $x_1 < x_*$ then $x_{n} < x_{n+1}  < x_*$
for all $n\ge1$, or if $x_1 > x_*$, then $x_n > x_{n+1} > x_*$
for all $n\ge1$
\cite[Remark 2.6]{ackleh2009classical}.

\begin{theorem}\label{thm:worstcaseextension}
Let $\cf$ have a worst case lower bound given
by~\eqref{eq:lowerbound} with
$\alpha>1$ and $0<m\le M<\infty$.
Suppose that there also exists a uniformly rate optimal sequence
$\bsx_i$ satisfying~\eqref{eq:upperbound}. 
If $\rho=\rho_k=n_{k+1}/n_k$, then 
\begin{align}
\label{eq:rhobound}
\rho \ge 1 + \biggl[\frac{m}M(1+\rho^{1-\alpha})^{-1}\biggr]^{1/(\alpha-1)}
> 1 + \Bigl(\frac{m}{2M}\Bigr)^{1/(\alpha-1)}.
\end{align}
\end{theorem}
\begin{proof}
From the lower bound~\eqref{eq:lowerbound}, there is an $f\in\cf$ with
\begin{align}
m(n_{k+1}-n_k)^{-\alpha}
& \le
\biggl|
\frac1{n_{k+1}-n_k}\sum_{i=n_k+1}^{n_{k+1}} (f(\bsx_i)-\mu)
\biggr|\notag\\
& =
(n_{k+1}-n_k)^{-1}
\biggl|
\sum_{i=1}^{n_{k+1}} (f(\bsx_i)-\mu)-
\sum_{i=1}^{n_{k}} (f(\bsx_i)-\mu)
\biggr|\notag\\
& \le M(n_{k+1}-n_k)^{-1}(n_{k+1}^{1-\alpha} +n_k^{1-\alpha} ),\label{eq:keyineq}
\end{align}
where we have applied upper bounds from~\eqref{eq:upperbound}.
Writing $\rho = n_{k+1}/n_k$ and rearranging~\eqref{eq:keyineq},
yields the first inequality in~\eqref{eq:rhobound}.

Next we define $g(\rho) = 1+( (m/M)(1+\rho^{1-\alpha})^{-1})^{1/(\alpha-1)}$,
the middle quantity in~\eqref{eq:rhobound}.
It has derivative
$$
g'(\rho) = \Bigl(\frac{m}M\Bigr)^{1/(\alpha-1)}\bigl( 1+\rho^{1-\alpha}\bigr)^{-1/(\alpha-1)-1} \rho^{-\alpha}
$$
which is positive for $\rho\in[1,2]$. Therefore $1<g(1) \le g(\rho)\le g(2)<2$
and so $g$ maps $[1,2]$ into $[1,2]$. 
Because $g'\le\lambda\equiv (m/M)^{1/(\alpha-1)}2^{-1/(\alpha-1)-1}$
the Lipschitz constant for $g$ is at most $\lambda<1$.
Thus $g$ has a unique fixed point $\rho_*\in(1,2)$ with $g(\rho_*)=\rho_*$.
If $\rho$ satisfies the first inequality in~\eqref{eq:rhobound},
then $\rho\ge\rho_*$.  By the monotone convergence of
this iteration, $\rho_*>g(1)=1+(m/(2M))^{1/(\alpha-1)}$.
\end{proof}

A rate optimal sequence $n_k$ must be spread out at
least geometrically, with a ratio $n_{k+1}/n_k\ge \rho$
where $\rho$ is now the smallest number satisfying the
first inequality in~\eqref{eq:rhobound}.
This $\rho$ depends on $m/M$ and on $\alpha$.
By definition $\rho >1$. Also we can find by
inspection that $\rho=2$ satisfies the
first inequality in~\eqref{eq:rhobound} and so
Theorem~\ref{thm:worstcaseextension} never
rules out a doubling of the sample size.
The lower bound on the critical extension factor
$n_{k+1}/n_k$ always satisfies $1<\rho<2$.

Figure~\ref{fig:extensionbound}
shows this extension bound as a function of 
$\alpha$ for various levels of the ratio $m/M$.
The values there were computed
via Brent's algorithm~\citep{bren:1973}.
The case $m/M =1$ is of special interest.  It describes
a rate-optimal rule that also attains the
optimal constant. That case has the highest bound
on the extension factor.

\begin{figure}
\centering
\includegraphics[width=\hsize]{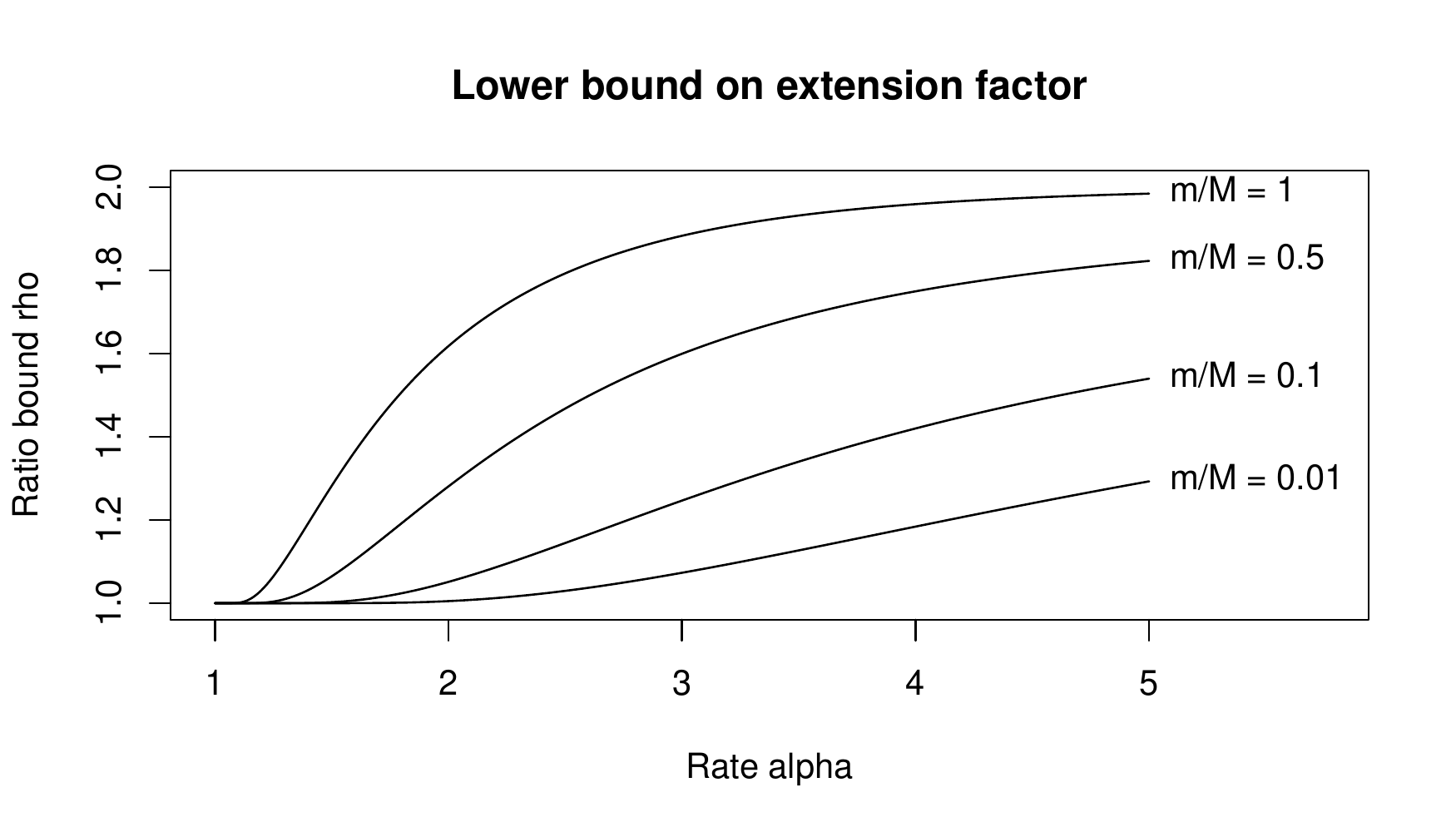}
\caption{\label{fig:extensionbound}
For worst case rates $n^{-\alpha}$ this figure shows
a lower bound on the extension factor $n_{k+1}/n_k$
in a uniformly rate-optimal integration sequence.
}
\end{figure}

If the rate $n^{-\alpha}$
is generalized to $n^{-\alpha}\log(n)^\beta$
for $\alpha>1$ and $\beta>0$, then a geometric
spacing is still necessary.
For any $1<\gamma<\alpha$ and large enough $n_k$
it is necessary to have
$\rho_k = n_{k+1}/n_k \ge 1+(m/2M)^{1/(1-\gamma)}$.

\section{The root mean square error setting}\label{sec:georms}

The class $\cf$ of real-valued functions
on $[0,1]^d$ has a superlinear root mean square lower bound if
\begin{align}\label{eq:lowerboundrms}
\sup_{f\in\cf}\,
\e\biggl(
\biggl|  \frac1n\sum_{i=1}^nf(\bsx_i)-\mu(f)\biggr|^2\biggr)^{1/2}
> mn^{-\alpha}
\end{align}
holds for some $\alpha>1$, $m>0$,
all $n\ge1$ and any random $\bsx_i\in[0,1]^d$.
Here $\e(\cdot)$ denotes expectation with respect to
the randomness in $\bsx_i$.
The sequence of random $\bsx_i\in[0,1]^d$ is uniformly
rate-optimal for this class if there is
a sequence of sample sizes $n_1<n_2<\cdots$, for which
\begin{align}\label{eq:upperboundrms}
\sup_{f\in\cf}\,
\e\biggl(\biggl|  \frac1n\sum_{i=1}^nf(\bsx_i)-\mu(f)\biggr|\biggr)^{1/2}
\le Mn^{-\alpha},\quad\forall n\in\{n_1,n_2,\dots\}
\end{align}
holds, where $m\le M<\infty$.

The same theorem holds for the root mean square
error case as holds for the worst case. 
It is not necessary to assume that any of $f(\bsx_i)$
are unbiased or to make any assumption about the
correlation structure among the $f(\bsx_i)$.
We only need to square one of the identities
in Theorem~\ref{thm:worstcaseextension} and then
use standard moment inequalities from probability theory.

\begin{theorem}\label{thm:rmscaseextension}
Let $\cf$ have a root mean square lower bound given
by~\eqref{eq:lowerboundrms} with
$\alpha>1$ and $0<m\le M<\infty$.
Suppose that there also exists a uniformly rate optimal sequence
of random $\bsx_i$ satisfying~\eqref{eq:upperboundrms}. 
If $\rho=\rho_k=n_{k+1}/n_k$, then 
\begin{align}
\label{eq:rhoboundrms}
\rho \ge 1 + \biggl[\frac{m}M(1+\rho^{1-\alpha})^{-1}\biggr]^{1/(\alpha-1)}
\ge 1 + \Bigl(\frac{m}{2M}\Bigr)^{1/(\alpha-1)}.
\end{align}
\end{theorem}
\begin{proof}
To shorten some expressions, let $\Delta_k = n_{k+1}-n_k$
and recall that $\eta_n = (1/n)\sum_{i=1}^n(f(\bsx_i)-\mu(f))$.
We begin with the identity,
\begin{align}\label{eq:basicident}
\frac1{\Delta_k}
\sum_{i=n_k+1}^{n_{k+1}} (f(\bsx_i)-\mu)
=
\frac1{\Delta_k}
\bigl( n_{k+1}\eta_{n_{k+1}} - n_k\eta_{n_k}\bigr).
\end{align}
The expected square of the left side of~\eqref{eq:basicident}
is no smaller than $m^2/\Delta_k^{2\alpha}$.
The expected square of the right side of~\eqref{eq:basicident} is
\begin{align*}
&\phantom{\le}\,
\frac1{\Delta_k^2}\Bigl[
n_{k+1}^2\e( \eta_{n_{k+1}}^2)
-2n_{k+1}n_k\e(\eta_{n_{k+1}}\eta_{n_k})
+n_{k}^2\e( \eta_{n_{k}}^2)\Bigr]\\
&\le\frac1{\Delta_k^2}\Bigl[
n_{k+1}^2\e( \eta_{n_{k+1}}^2)
+2n_{k+1}n_k\sqrt{\e(\eta_{n_{k+1}}^2)\e(\eta_{n_k}^2)}
+n_{k}^2\e( \eta_{n_{k}}^2)\Bigr]\\
&\le
\frac{M^2}{\Delta_k^2}\Bigl[
n_{k+1}^{2(1-\alpha)}
+2n_{k+1}^{1-\alpha}n_k^{1-\alpha}
+n_{k}^{2(1-\alpha)}\Bigr].
\end{align*}
As a result,
\begin{align}\label{eq:squaredident}
\frac{m^2}{\Delta_k^{2\alpha}}\le \frac{M^2}{\Delta_k^2}\Bigl[
n_{k+1}^{1-\alpha}+n_{k}^{1-\alpha}\Bigr]^2.
\end{align}
Taking the square root of both sides of~\eqref{eq:squaredident}
and rearranging, yields~\eqref{eq:keyineq}, from which the
theorem follows just as it did for the worst case analysis.
\end{proof}

\section{Discussion}\label{sec:discussion}

We have found a constraint on the spacings of
a rate-optimal equal-weight extensible MC or
QMC sequence.  The constraint only applies
when the convergence is better than $O(1/n)$.

We can use that constraint in reverse as follows.
Suppose that a rate-optimal sequence $n_k$
includes sample sizes with $n_{k+1}/n_k=\rho\in(1,2)$
and attains the error rate $O(n^{-\alpha})$ for $\alpha>1$.
Then from~\eqref{eq:rhobound} we obtain
$$
\frac{M}m\ge \frac{ (1+\rho^{1-\alpha})^{-1/(\alpha-1)}}{\rho-1}.
$$
As we approach an arithmetic progression
by letting $\rho\downarrow 1$, the inefficiency in the constant
factor increases without bound.

The constraint only applies to rate-optimal sequences.
In particular it does not apply to an extensible sequence that
may be inefficient by a logarithmic factor.

\section*{Acknowledgments}

I thank Alex Kreinin and Sergei Kucherenko
for pointing out to me the
elegant argument in the Appendix of \cite{sobo:1998},
and Erich Novak for sharing his slides from
MCQMC 2014.
This work was supported by grants DMS-0906056 
and DMS-1407397 of the U.S. National Science Foundation.

\bibliographystyle{apalike}
\bibliography{qmc}
\end{document}